\documentclass{amsart}

\usepackage{bm,amsmath,amssymb}
\usepackage{graphicx}

\newtheorem{theorem}{Theorem}
\newtheorem{lemma}[theorem]{Theorem}
\newtheorem{remark}[theorem]{Remark}

\numberwithin{equation}{section}

\newcommand{\ang}[1]{\langle #1 \rangle}

\newcommand\enorm[1]{|\!|\!| #1|\!|\!|}	
\newcommand{\ngrad}{\bm\partial_{\bm n}}
\newcommand{\dK}{\partial K}
\newcommand{\rd}{\partial}
\newcommand{\R}{{\mathbb R}}
\newcommand{\Proj}{{\sf P}}
\newcommand{\T}{\mathcal{T}}
\newcommand{\dT}{\partial\mathcal{T}}
\newcommand{\E}{\mathcal{E}}

\newcommand{\diam}{{\rm diam}}
\newcommand{\sumK}{\sum_{K \in \T_h}}
\newcommand{\I}{{\sf I}}

\newcommand{\dive}{\mathrm{div}}
\newcommand{\vech}[1]{\vec{\widehat #1}}
\renewcommand{\vec}{\bm}

\newcommand{\tvec}[1]{\widetilde{\vec #1}}
\newcommand{\jump}[1]{[\![#1]\!]}
\newcommand{\V}{\vec V_h^{k+1}}
\newcommand{\hV}{\vech V_h^{k}}
\newcommand{\tV}{\tvec V_h^{k+1}}

\newcommand{\jumpuh}{|(\vec u_h^\tau, \vech u_h^\tau)|_{\rm j}}

\title
{Analysis of a reduced-order HDG method for the Stokes equations
}
\author{Issei Oikawa}    
         
\date{}

\begin{document} 
\maketitle

\begin{abstract} 
In this paper, we  analyze 
 a hybridized discontinuous Galerkin(HDG) method
  with reduced stabilization for the Stokes equations.
 The reduced stabilization enables us to reduce
 the number of  facet unknowns and improve the computational efficiency of the method. 
 We provide optimal error estimates in an energy and $L^2$ norms.
 It is shown that the reduced method with the lowest-order approximation
  is closely related to the nonconforming Crouzeix-Raviart finite element method.
 We also prove that
  the solution of the reduced method 
  converges to the nonconforming Gauss-Legendre finite element
   solution as a stabilization parameter $\tau$ tends to infinity 
   and that  the convergence rate is $O(\tau^{-1})$.
   
\end{abstract}



\section{Introduction}
 The aim of this paper is to propose and analyze a reduced-order hybridized discontinuous Galerkin(HDG) method for the Stokes equations with no-slip boundary condition:   
 \begin{equation} \label{stokes}
 \begin{aligned}
  -\Delta \vec u + \nabla p &= \vec f  \text{ in } \Omega, \\
  \dive \vec u & = 0      \text{ in } \Omega,\\ 
  \vec u &= \vec 0  \text{ on } \rd \Omega,
 \end{aligned}
 \end{equation}
 where $\Omega \subset \R^d (d=2,3)$ is a convex polygonal or polyhedral domain and $\vec f \in \vec L^2(\Omega) := [L^2(\Omega)]^d$ is a given function.
 For the Stokes problem, various HDG methods were already proposed and studied
 \cite{CCS2006,CoGo2009,CNP2010,NPC2010,CGNPS2011,CoCu2012a,CoCu2012b,CoSa2014,Lehrenfeld2010,EgWa2013}.
 We also refer to \cite{CoShi2014} for an overview.
 The method we investigate in this paper is the HDG-IP method proposed by Egger and Waluga in \cite{Lehrenfeld2010}.
 The HDG-IP method is based on the gradient-velocity-pressure formulation of the Stokes equations.
 We remark that the HDG method of the local discontinuous Galerkin type\cite{NPC2010,CGNPS2011} is
 close to the HDG-IP with a slight difference of a numerical flux. 
         
 A reduced stabilization was  introduced to the DG methods and was analyzed for elliptic problems\cite{BuSt2008,BuSt2009}.
 In \cite{BCJ2012}, Becker et al. studied the reduced-order DG method for the Stokes equations and 
 analyzed the limit case as a stabilization parameter tends to infinity.
  In \cite{Lehrenfeld2010}, Lehrenfeld proposed a reduced-order HDG method for the Poisson equation and the Stokes equations. 
   Lehrenfeld also remarked that the convergence rate of the method is optimal,  however, error analysis was not presented. 
   In \cite{OikawaToAppear}, for the Poisson problem, the author
   provided the optimal error estimates and showed the  reduced-order HDG method with the lowest-order approximation
   is closely related to the nonconforming Crouzeix-Raviart finite element method.  
  Recently, Qiu and Shi analyzed the reduced methods 
  for linear elasticity problems\cite{QiuShi2015} and convection-diffusion equations\cite{QiuShi2014Submitted}.  

   The reduced-order HDG method uses polynomials of degree $k+1$ and $k$ to
   approximate element and hybrid unknowns, respectively, whereas
   the standard HDG method uses polynomials of degree $k$ for both unknowns.
    Although both the methods have
    the same number of globally coupled degrees of freedom,
   the reduced method can provide 
   higher order convergence than the standard method.
   This is the main advantage of the reduced-order HDG method.
   For the Stokes problem,  when we use polygonal or polyhedral elements, the reduced method
    is indeed better than the standard method in term of convergence orders. 
   The standard method using polynomials of degree $k$  
   for all unknowns  obtains the suboptimal convergence; the orders are  $k+1$ for velocity  without postprocessing and $k+1/2$ for the gradient and pressure, according to \cite{CoShi2014}.  
   In contrast, the reduced method uses polynomials of degree $k+1$ for velocity and polynomials of degree $k$ for the hybrid part of velocity and pressure and can achieve  the optimal order convergence.

 In this paper, we provide optimal error estimates of the reduced method
  for the Stokes problem.   
 Since we need to use a weaker energy norm
 in our analysis, it is necessary to modify the error analysis of the standard method.
 We note that the main difficulties can be overcome by the techniques 
   used in the author's previous work \cite{OikawaToAppear}
 and the discrete inf-sup condition proved by Egger and Waluga \cite{EgWa2013}.
 We also show a relation between the reduced method and 
 the Gauss-Legendre element(see \cite{BBF2013,StBa2006} for example).
 It is proved that the hybrid part of velocity and the pressure of the reduced-order HDG method with the lowest-order approximation 
 coincides with those of the nonconforming Crouzeix-Raviart finite element solution. 
 In the limit case as the stabilization parameter $\tau$ 
 tends to infinity, the solution of the reduced method 
 converges to that of the nonconforming Gauss-Legendre method.
 The convergence rate is estimated to be $O(\tau^{-1})$.
 This result is inspired by \cite[Theorem 3]{BCJ2012}, 
 however, our proof is completely different and novel.

 The rest of this paper is organized as follows. 
 Section 2 is devoted to the preliminaries.
 In Section 3, we introduce a reduced stabilization and present a reduced HDG method.
 In Section 4, we provide a priori error estimates in an energy and $L^2$ norms.
 In Section 5, some relations between the nonconforming Gauss-Legendre 
  finite element method and the reduced method are shown.
 In Section 6, numerical results are presented to confirm our theoretical results.

\section{Preliminaries}
 \subsection{Meshes and function spaces}
Let $\{\T_h\}_h$ be a family of shape-regular triangulations of $\Omega$ and
 define $\Gamma_h = \bigcup_{K \in \T_h} \dK$.
 Let $\E_h$ be the set of all edges in $\T_h$.
 The mesh size of $\T_h$ is denoted by $h$, namely $h := \max_{K \in \T_h} h_K$,
  where $h_K = \diam K$. The length of an edge $e \in \E_h$ is denoted by $h_e$.

 We use the usual Lebesgue and Sobolev spaces; $L^2(\Omega)$, $L^2(\Gamma_h)$ and $H^m(\Omega)$,
  and also $L^2_0(\Omega) = \{ q \in L^2(\Omega) : \int_\Omega q dx = 0\}$.
We introduce piecewise Sobolev spaces
$
  H^m(\T_h) = \{ v \in L^2(\Omega) : v|_K \in H^m(K)\ \forall K \in \T_h\}.
$
 For vector-valued function spaces, we write them in bold,
  such as $\vec L^2(\Omega) = [L^2(\Omega)]^d$ and $\vec H^m(\Omega) = [H^m(\Omega)]^d$.
 The usual $L^2$ inner product is denoted by
 $(\cdot, \cdot)_\Omega$.
  Let us define the piecewise inner products by 
  \[ 
  (u, v)_{\T_h}
  = \sumK \int_K u v dx, \quad \ang{u,v}_{\dT_h} = \sumK \int_{\dK} uv ds.
  \]
  Let $P_k(\T_h)$ and $P_k(\E_h)$ denote the space of \emph{element-wise}
   and \emph{edge-wise} polynomials of degree $k$, respectively.
  We employ $\vec V_h^{k+1} = \vec P_{k+1}(\T_h)$,
  $\vech V_h^{k} = \vec P_{k}(\E_h) \cap 
   \{ \vech v \in \vec L^2(\Gamma_h) : \vech v = \vec 0$ 
   on  $\partial\Omega \}$
  and $Q_h^{k} = P_{k}(\T_h) \cap L^2_0(\Omega)$ as finite element spaces,
  which we call $P_{k+1}$--$P_{k}/P_{k}$ approximation.
  The $L^2$-projection from  
  $\prod_{K \in \T_h} \vec L^2(\dK)$ onto $\prod_{K \in \T_h} \vec P_{k}(\dK)$ 
   is denoted by $\Proj_{k}$, and $\I$ stands for the identity operator.
      
  \subsection{Norms and seminorms}  
 As usual, we use the Sobolev norms 
 $|\vec v|_m = |\vec v|_{\vec H^m(\Omega)}$ and $\|\vec v\|_{m,D} = \|\vec v\|_{\vec H^m(\Omega)}$ for a domain $D$. The $L^2$-norm is denoted by $\|\vec v\| = \|\vec v\|_{0, D} = \|\vec v\|_{\bm L^2(D)}$.
 The energy norms are defined as follows: for $(\vec v, \vech v) \in \vec H^2(\T_h) \times \bm L^2(\Gamma_h)$,
 \begin{align*}
    \enorm{(\vec v, \vech v)}^2 & = |\vec v|_{1,h}^2 + |\vec v|_{2,h}^2 +  |(\vec v, \vech v)|^2_{\rm j}, \\
    \enorm{(\vec v, \vech v)}^2_\tau & = |\vec v|_{1,h}^2 + |\vec v|_{2,h}^2  +  |(\vec v, \vech v)|^2_{\rm j, \tau},\\
    \enorm{(\vec v, \vech v)}^2_h & = |\vec v|_{1,h}^2  +  |(\vec v, \vech v)|^2_{\rm j}, \\
    \enorm{(\vec v, \vech v)}^2_{h,\tau} & = |\vec v|_{1,h}^2  +  |(\vec v, \vech v)|^2_{\rm j, \tau},
\end{align*}
   where
   \begin{align*}
    |(\vec v, \vech v)|^2_{\rm j} &= 
     \sumK \sum_{e\subset \dK} \frac{1}{h_e} \|\Proj_{k}(\vech v - \vec v)\|_{0,e}^2, \\
    |(\vec v, \vech v)|^2_{\rm j,\tau}, &= 
    \tau |(\vec v, \vech v)|^2_{\rm j} \\ 
    |\vec v|_{1,h}^2 & = \sumK  |\vec v|_{1,K}^2, \\
    |\vec v|_{2,h}^2 &=\sumK  h_K^2|\vec v|_{2,K}^2. 
    \end{align*}
    The symbol $\tau$ is a stabilization parameter which will be defined in Section \ref{sec:rhdg}.
    The parameter-free energy norms $\enorm{\cdot}$
    and  $\enorm{\cdot}_h$ are used to 
    analyze the convergence rate with respect to the mesh size $h$.
    We need the parameter-dependent energy norms in order to 
    analyze the proposed method when $\tau \to \infty$.
    We also use the stronger $L^2$ norm 
 \[
  \|q\|_{h}^2 = \|q\|^2 + \sum_{K \in \T_h} h_K^2 |q|_{1,K}^2.
 \] 
 By the inverse inequality\cite{Arnold1982}, we see that the two energy norms are equivalent to each other
  on $\V \times \hV$, i.e.,  
 \begin{align} \label{en-equiv}
    &\enorm{(\vec v_h, \vech v_h)}_h 
	 \le \enorm{(\vec v_h, \vech v_h)}
	 \le C\enorm{(\vec v_h, \vech v_h)}_h , \\
	&\enorm{(\vec v_h, \vech v_h)}_{h,\tau} 
	 \le \enorm{(\vec v_h, \vech v_h)}_\tau
     \le C\enorm{(\vec v_h, \vech v_h)}_{h,\tau} 
 \end{align} 
 for some constant $C>0$ independent of $h$. 
 Similarly, it holds that $\|q_h\| \le \|q_h\|_{h} \le C \| q_h\|$ 
 for all $q_h \in Q_h^{k}$.  
 In the following, the symbol $C$ will stand for a generic constant 
 independent of the mesh size $h$ and the stabilization parameter $\tau$. 
     
\subsection{Approximation property}
 The approximation property in the energy norm holds 
  as well as in the standard energy norm.
\begin{lemma}
 If $\vec \psi \in  \vec H^{k+2}(\Omega)$
  and $\pi \in H^{k+1}(\Omega)$, then we have
 \begin{align}  
 \inf_{(\vec v_h, \vech v_h)  \in \vec V_h^{k+1} \times \vech V_h^{k}}
  \enorm{(\vec \psi - \vec v_h,\vec \psi|_{\Gamma_h} - \vech v_h)}
   &\le Ch^{k+1} |\vec \psi|_{\vec H^{k+2}(\Omega)},
  \\ 
  \inf_{q_h \in Q_h^{k}} \| \pi - q_h\| &\le C h^{k+1} |\pi|_{H^{k+1}(\Omega)}. 
 \end{align}

 \end{lemma}
 \begin{proof}
  We refer to \cite{OikawaToAppear}.   
 \end{proof}

\section{A reduced-order HDG method} \label{sec:rhdg}
 In this section, we present a reduced-order HDG method
  based on the HDG method proposed by Egger and Waluga in \cite{EgWa2013}. 
 By taking the $L^2$-projection onto 
 the polynomial space of lower degree by one in the stabilization term of the standard method, we 
 obtain the reduced-order HDG method: find $(\vec u_h, \vech u_h, p_h) 
 \in \vec V_h^{k+1} \times \vech V_h^{k} \times Q_h^{k}$ such that
 \begin{subequations}\label{rhdg}
 \begin{align}
   a_h(\vec u_h, \vech u_h; \vec v_h, \vech v_h)    
   +  b_h(\vec v_h, \vech v_h;  p_h)
   &= (\vec f, \vec v_h)_\Omega
   & \forall& (\vec v_h, \vech v_h) \in \vec V_h^{k+1} \times \vech V_h^{k},   
   \label{rhdg1}
   \\
   b_h(\vec u_h, \vech u_h; q_h) &= 0 
   & \forall& q_h \in Q_h^{k},
   \label{rhdg2}
 \end{align}
 \end{subequations}
 where the bilinear forms are given by
 \begin{align*}
   a_h(\vec u_h, \vech u_h; \vec v_h, \vech v_h)
    &= (\nabla \vec u_h, \nabla \vec v_h)_{\T_h} 
       + \ang{\ngrad \vec u_h, \vech v_h - \vec v_h}_{\dT_h} 
      + \ang{\ngrad \vec v_h, \vech u_h - \vec u_h}_{\dT_h} \\
    &  \quad + \ang{\tau h_e^{-1} \Proj_{k}(\vech u_h - \vec u_h), \Proj_{k}(\vech v_h - \vec v_h)}_{\dT_h}, \\  
    b_h(\vec v_h, \vech v_h;  p_h) 
     & = - (\dive \vec v_h, p_h)_{\T_h} - \ang{\vech v_h - \vec v_h, p_h\vec n}_{\dT_h}.
 \end{align*}
 Here $\tau$ is a stabilization parameter
  assumed to be greater than or equal to one and sufficiently large.
 We recall that $\Proj_{k}$, which is defined in Section 2.1., is the $L^2$-projection
 onto the edge-wise polynomial space of degree $k$.

\begin{remark}
  In the two-dimensional case, 
  we can easily implement the reduced method by using a reduced-order quadrature formula
  in the computations of the reduced stabilization term, see \cite[Lemma 5]{OikawaToAppear} for details.
\end{remark}
   
\section{Error analysis}
In this section, we provide the optimal error estimates of the method in both the energy and $L^2$
 norms.
  To do that, we  first show the consistency of the method,  the boundedness of $a_h$ and $b_h$,  and the coercivity of $a_h$.  
  In addition,  the discrete inf-sup condition of $b_h$
  is proved based on the results of \cite{EgWa2013}. 

\subsection{Consistency}
 We state the consistency and adjoint consistency of the method.
\begin{lemma}
 Let $(\vec u, p)$ be the exact solution of the Stokes equations \eqref{stokes}.
 Then we have 
 \begin{equation} 
 \begin{aligned}
    a_h(\vec u, \vec u|_{\Gamma_h}; \vec v_h, \vech v_h)
       + b_h(\vec v_h, \vech v_h; p) &= (\vec f, \vec v_h)_\Omega
       &\forall &(\vec v_h, \vech v_h) \in \V\times \hV, \\
     b_h(\vec u, \vec u|_{\Gamma_h};  q_h) &= 0
      &  \forall& q_h \in Q^{k}_h. 
 \end{aligned}
 \label{cons}
 \end{equation}
\end{lemma}
\begin{proof}
 Since $\vec u - \vec u|_{\Gamma_h} = \vec 0$ on $\Gamma_h$,
 we can easily see that the consistency  \eqref{cons} holds.   
\end{proof}
 
 Let $(\vec u_h, \vech u_h, p_h) \in \V \times \hV \times Q_h^{k}$
 be the solution of the  method \eqref{rhdg}. 
  From the consistency, the Galerkin orthogonality  follows immediately:
 \begin{equation}\label{gorth}
 \begin{aligned}
    a_h(\vec u - \vec u_h, \vec u|_{\Gamma_h} - \vech u_h; \vec v_h, \vech v_h)
       + b_h(\vec v_h, \vech v_h; p - p_h) &= 0 
       &\forall& (\vec v_h, \vech v_h) \in \V \times \hV, \\
     b_h(\vec u-\vec u_h, \vec u|_{\Gamma_h} -\vech u_h;  q_h) &= 0
      &  \forall& q_h \in Q^{k}_h.
 \end{aligned}
 \end{equation}
 Due to the symmetricity of $a_h$,  we readily  see that the adjoint consistency also holds:
     \begin{equation} 
     \label{acons}
 \begin{aligned}
    a_h(\vec v_h, \vech v_h; \vec u, \vec u|_{\Gamma_h})
       + b_h(\vec v_h, \vech v_h; p) &= (\vec f, \vec v_h)_\Omega 
       &\forall& (\vec v_h, \vech v_h) \in \V \times \hV, \\
     b_h(\vec u, \vec u|_{\Gamma_h};  q_h) &= 0
      &  \forall& q_h \in Q_h^{k}. 
 \end{aligned}
 \end{equation}
  
\subsection{Boundedness and coercivity}
 We first prove the boundedness of $a_h$ and $b_h$.
\begin{lemma} \label{bdd-ah}
 Let $(\vec\xi,\vech \xi) = (\vec{w} + \vec w_h, \vec{w}|_{\Gamma_h} + \vech w_h)$
 and $(\vec\eta,\vech \eta) = (\vec{v} + \vec v_h , \vec{v}|_{\Gamma_h} + \vech v_h)$,
 where $(\vec w_h, \vech w_h)$,
 $(\vec v_h , \vech v_h) \in  \V \times \hV$ and  
 $\vec{w}, \vec v \in \vec H^1_0(\Omega)$.  
 Then there exists a constant $C>0$ independent of $h$ and $\tau$ such that
 \begin{equation} \label{bdd-a}
  |a_h(\vec\xi, \vech \xi;  \vec\eta, \vech \eta)|
          \le C \tau\enorm{(\vec\xi, \vech \xi)}
                \enorm{(\vec\eta, \vech \eta)}.
 \end{equation}
 With respect to the parameter-dependent energy norm, we have 
 \begin{equation} \label{bdd-a-tau}
  |a_h(\vec\xi, \vech \xi; \vec\eta, \vech \eta)|
          \le C \enorm{(\vec\xi, \vech \xi)}_\tau
                \enorm{(\vec\eta, \vech \eta)}_\tau.
 \end{equation}
\end{lemma}
\begin{proof} 
 By the Schwarz inequality, the first term of  $a_h$ is bounded as
 \begin{align}
  |(\nabla \vec \xi, \nabla \vec \eta)_{\T_h}| \le 
  |\vec \xi|_{1,h} |\vec\eta|_{1,h}.
 \end{align}
 We estimate the second term. Note that
 \begin{equation}\label{bdd1} \begin{aligned}
 \ang{\ngrad \vec \xi, \vech \eta - \vec \eta}_{\dT_h} 
  &=\ang{\ngrad \vec{\xi}, \Proj_{k}(\vech \eta - \vec \eta)}_{\dT_h} 
   +\ang{\ngrad \vec{\xi}, (\I-\Proj_{k})(\vech \eta - \vec \eta)}_{\dT_h} 
     \\
  &=\ang{\ngrad \vec{\xi}, \Proj_{k}(\vech \eta - \vec \eta)}_{\dT_h} 
    -\ang{\ngrad \vec{\xi}, (\I-\Proj_{k})\vec v_h}_{\dT_h}. 
 \end{aligned}\end{equation}
 Since $\ang{\ngrad \vec{w}, (\I-\Proj_{k})\vec{v}}_{\dT_h} = 
 \ang{\ngrad \vec w_h, (\I-\Proj_{k})\vec{v}}_{\dT_h} = 0$, it follows that 
 \[
  \ang{\ngrad \vec{\xi}, (\I-\Proj_{k})\vec{v}}_{\dT_h}=0.
 \]
 Using this, we deduce that 
 $\ang{\ngrad \vec{\xi}, (\I-\Proj_{k})\vec v_h}_{\dT_h} 
   = \ang{\ngrad \vec{\xi}, (\I-\Proj_{k})\vec \eta}_{\dT_h}$. 
 Then we have
 \begin{equation}\label{bdd4} 
 \begin{aligned} 
  |\ang{\ngrad \vec \xi, \vech \eta - \vec \eta}_{\dT_h}|
  &= |\ang{\ngrad \vec\xi, \Proj_{k}(\vech v_h - \vec v_h)
  -(\I - \Proj_{k}) \vec\eta}_{\dT_h}| \\
  &\le C \max\{1, \tau^{-1/2}\}(|\vec \xi|_{1,h}^2 
    +h^2|\vec \xi|_{2,h}^2)^{1/2}
   (|(\vec\eta, \vech\eta)|_{\rm j, \tau}^2  +  |\vec\eta|_{1,h}^2)^{1/2} \\
  &\le C\enorm{(\vec\xi, \vech \xi)}_\tau
        \enorm{(\vec\eta, \vech \eta)}_\tau,
 \end{aligned}    
 \end{equation} 
 where we have used the trace inequality  and the following estimate (see 
 \cite{OikawaToAppear} for the proof)
 \begin{equation} \label{ineq1}
     |(\vec\eta, \vech\eta)|_{\rm j} \le C|\vec\eta|_{1,h}^2.
 \end{equation}
 The stabilization term is bounded as
 \begin{equation}\label{bdd5} 
 \begin{aligned}
  |\ang{\tau h_e^{-1} \Proj_{k}(\vech \xi - \vec \xi), \Proj_{k}(\vech \eta - \vec \eta)}_{\dT_h}|
  &\le |(\vec \xi, \vech \xi)|_{\rm j, \tau} |(\vec \eta, \vech \eta)|_{\rm j, \tau}.
 \end{aligned}
 \end{equation} 
 From \eqref{bdd1}, \eqref{bdd4} and \eqref{bdd5},
 we obtain the boundedness \eqref{bdd-a-tau}. 
 With a slight modification, we can also show that 
 \eqref{bdd-a}.   
\end{proof}

\begin{lemma}[Boundedness of $b_h$] \label{bdd-bh}
 Let $(\vec\xi,\vech \xi)$ be the same as in Lemma \ref{bdd-ah}
 and $r = q + q_h$ with $q \in H^1(\T_h) \cap L^2_0(\Omega)$
 and $q_h \in Q_h^{k}$.  
 Then there exists a constant $C>0$ independent $h$ and $\tau$ such that
 \begin{equation} \label{bdd-b}
  |b_h(\vec \xi, \vech \xi; r)| \le 
  C \enorm{(\vec \xi, \vech \xi)} \|r\|_{h}. 
 \end{equation}
 In particular, in the case of $r = q_h \in Q_h^{k}$, we have
 \begin{equation} \label{bdd-b-qh}
   |b_h(\vec \xi, \vech \xi; q_h)| \le 
   C \enorm{(\vec \xi, \vech \xi)} \|q_h \|. 
 \end{equation}         
\end{lemma}
\begin{proof}
  By the Schwarz inequality, we have
  \begin{equation}\label{bdd-b1} 
   \begin{aligned} 
    |(\dive \vec \xi, r)_{\T_h}| \le |\vec \xi|_{1,h} \|r\|.
   \end{aligned}
  \end{equation} 
  Next, we estimate the second term of the bilinear form $b_h$.
  In a similar manner of \eqref{bdd1}, we have
  \begin{equation}\label{bdd-b2} 
  \begin{aligned} 
   \ang{\vech\xi  - \vec \xi, r\vec n}_{\dT_h}
   &=\ang{\Proj_{k}(\vech\xi  - \vec \xi), r\vec n}_{\dT_h}
    +\ang{(\I - \Proj_{k})(\vech\xi  - \vec \xi), r\vec n}_{\dT_h}\\
    &=\ang{\Proj_{k}(\vech\xi  - \vec \xi), r\vec n}_{\dT_h}
    -\ang{(\I - \Proj_{k})\vec w_h, r\vec n}_{\dT_h}.
   \end{aligned}
    \end{equation} 
   Since $\ang{(\I - \Proj_{k})\vec {w}, r\vec n}_{\dT_h} =0$,
    we get 
    \begin{align*}
    \ang{(\I - \Proj_{k})\vec w_h, r\vec n}_{\dT_h}
     =  \ang{(\I - \Proj_{k})(\vec{w} + \vec w_h), r\vec n}_{\dT_h}
     = \ang{(\I - \Proj_{k})\vec\xi, r\vec n}_{\dT_h}.
     \end{align*}
     Hence 
   \begin{equation*} 
    \begin{aligned}
    |\ang{\vech\xi  - \vec \xi, r\vec n}_{\dT_h}|
     &= |\ang{\Proj_{k}(\vech\xi  - \vec \xi) 
    -(\I - \Proj_{k})\vec \xi, r\vec n}_{\dT_h}| \\
     &\le C(|(\vec\xi, \vech\xi)|_{\rm j}^2+|\vec\xi|_{1,h}^2)^{1/2}
             \|r\|_{h},
    \end{aligned}
   \end{equation*} 
     where we have used the trace inequality and \eqref{ineq1}. 
  Consequently,  we obtain the inequality \eqref{bdd-b}.
  From the inverse inequality, \eqref{bdd-b-qh} follows immediately.   
   
\end{proof}

\begin{lemma}[Coercivity] \label{coer-ah}
 Assume that $\tau$ is sufficiently large.
  There exists a constant $C>0$ independent of $h$ and $\tau$ such that
 \begin{equation}\label{coer} 
  \begin{aligned} 
   a_h(\vec v_h, \vech v_h; \vec v_h, \vech v_h) \ge C\enorm{(\vec v_h, \vech v_h)}^2_\tau
   \quad \forall (\vec v_h, \vech v_h) \in \V \times \hV. 
  \end{aligned}
 \end{equation}
 In particular, we have
  \begin{equation}\label{coer-tau} 
  \begin{aligned} 
   a_h(\vec v_h, \vech v_h; \vec v_h, \vech v_h) \ge C\enorm{(\vec v_h, \vech v_h)}^2
   \quad \forall (\vec v_h, \vech v_h) \in \V \times \hV. 
  \end{aligned}
 \end{equation}
\end{lemma}
\begin{proof}
   We note that
   \begin{equation}\label{coer1} 
    \begin{aligned} 
      \ang{\ngrad{\vec v_h}, \vech v_h - \vec v_h}_{\dT_h}
      = \ang{\ngrad{\vec v_h}, \Proj_{k}(\vech v_h - \vec v_h)}_{\dT_h}.
    \end{aligned}
    \end{equation} 
    Then it follows that
    \begin{equation}\label{coer2} 
     \begin{aligned} 
     a_h(\vec v_h, \vech v_h; \vec v_h, \vech v_h) 
     \ge |\vec v_h|^2_{1,h} - 2 |\ang{\ngrad{\vec v_h}, \Proj_{k}(\vech v_h - \vec v_h)}_{\dT_h}|  
       + | (\vec v_h, \vech v_h)|_{\rm j, \tau}^2.
     \end{aligned}
     \end{equation} 
     By the trace and inverse inequalities and Young's inequality, we have
     \begin{align} \label{coer3}
     2 |\ang{\ngrad{\vec v_h}, \Proj_{k}(\vech v_h - \vec v_h)}_{\dT_h}|
     & \le
     C\left(
      \varepsilon |\vec v_h|^2_{1,h} + \varepsilon^{-1} \tau^{-1}|(\vec v_h, \vech v_h)|_{\rm j, \tau}^2
     \right)
     \end{align} 
         for any $\varepsilon > 0$.
     From \eqref{coer2} and \eqref{coer3}, it follows that
    \begin{equation}\label{coer4} 
     \begin{aligned} 
     a_h(\vec v_h, \vech v_h; \vec v_h, \vech v_h) \ge 
      (1-C\varepsilon)|\vec v_h|_{1,h}^2 
      + (1 - C\varepsilon^{-1}\tau^{-1}) 
      |(\vec v_h, \vech v_h)|_{\rm j, \tau}^2.
     \end{aligned}
    \end{equation} 
     We can take
     $\varepsilon = \tau^{-1/2}$ and deduce that, by assuming $\tau \ge 4C^2$,  
     \[
         a_h(\vec v_h, \vech v_h; \vec v_h, \vech v_h) \ge 
          \left(1 - C \tau^{-1/2}\right) \enorm{(\vec v_h, \vech v_h)}_{h,\tau}^2
          \ge \frac 1 2 \enorm{(\vec v_h, \vech v_h)}_{h,\tau}^2.
     \]
     By the inverse inequality, we have
      $\enorm{(\vec v_h, \vech v_h)}_\tau \le C' \enorm{(\vec v_h, \vech v_h)}_{h,\tau}$
       for some positive constant $C'$,
       which completes the proof.  
 \end{proof}

 \subsection{The discrete inf-sup condition}
 To prove the discrete inf-sup condition of the bilinear form $b_h$,
 we introduce a Fortin operator. The main idea and techniques 
 for constructing the Fortin operator
 are entirely based on \cite{EgWa2013}. 
 The global $L^2$-projection operators $\vec \Pi_h^k : \vec H^1_0(\Omega) \to \vec V_h^k$  and 
 $\vech \Pi_h^k : \vec H^{1}(\Omega) \to  \vech V_h^{k}$ 
 are defined by
 \begin{align*}
  (\vec \Pi_h^k \vec v)|_K &= \vec\Pi_K^{k} (\vec v|_K)
  \text{ for } K \in \T_h, \\
  (\vech \Pi_h^{k} \vec v)|_e &= \vech\Pi_e^{k} (\vec v|_e)
  \text{ for } e \in \E_h, 
 \end{align*}
 where $\vec\Pi_K^{k}$ and
 $\vech\Pi_e^k$ are the $L^2$-projections onto 
 $\vec P_k(K)$ and $\vec P_{k}(e)$, respectively. 
 We define the Fortin operator by
 \begin{align*}
  (\vec \Pi_h^{k+1}, \vech \Pi_h^{k}) : \vec H_0^1(\Omega)
  \to \vec V_h^{k+1} \times \vech V_h^{k}.
 \end{align*}
 In the following, we show this operator satisfies the Fortin properties.
   
\begin{lemma}
 For all $\vec v \in \vec H_0^1(\Omega)$,
 we have
 \begin{equation} \label{fortin1}
  b_h(\vec\Pi_h^{k+1} \vec v, \vech\Pi_h^{k} \vec v; q_h) = -(\dive \vec v, q_h)_\Omega \quad \forall q_h \in Q_h^{k}.
 \end{equation} 
 Moreover, there exists a constant $C>0$ such that
 \begin{equation} \label{fortin2}
  \enorm{(\vec\Pi_h^{k+1}\vec v, \vech\Pi_h^{k} \vec v)}_h \le C |\vec v|_{1,h}
  \quad \forall \vec v \in \vec H_0^1(\Omega).
 \end{equation} 
\end{lemma}
\begin{proof}
 First, we prove \eqref{fortin1}.
 Using the Green formula and the property of $L^2$ projection, we have
 \begin{equation}\label{fortin3} 
 \begin{aligned} 
   b_h(\vec\Pi_h^{k+1} \vec v, \vech\Pi_h^{k} \vec v; q_h)
   &= -(\dive \vec\Pi_h^{k+1} \vec v, q_h)_{\T_h} 
      - \ang{\vech\Pi_h^{k} \vec v- \vec\Pi_h^{k+1}\vec v, q_h \vec n}_{\dT_h} \\
   &= (\vec\Pi_h^{k+1} \vec v, \nabla q_h)_{\T_h}
      - \ang{\vech\Pi_h^{k} \vec v, q_h \vec n}_{\dT_h} \\
   &= ( \vec v, \nabla q_h)_{\T_h}
      - \ang{ \vec v, q_h \vec n}_{\dT_h} \\      
   &= -(\dive \vec v, q_h)_{\Omega}. 
 \end{aligned}
\end{equation} 
 Next, we prove \eqref{fortin2}.
 Note that $| \vec\Pi_h^{k+1} \vec v|_{1,h} \le |\vec v|_{1,h}$ and 
 \begin{equation}\label{fortin5} 
  \begin{aligned} 
   h_e^{-1/2}\|\Proj_{k}(\vech\Pi_h^{k}\vec v-\vec\Pi_h^{k+1} \vec v)\|_{0,e}
      &\le   h_e^{-1/2}\|\vec v -\vec\Pi_h^{k+1} \vec v\|_{0,e}\\
      &\le   C| \vec v|_{1,h}.
  \end{aligned}
 \end{equation} 
 From \eqref{fortin3} and \eqref{fortin5}, it follows that 
 \begin{equation}\label{fortin6} 
 \begin{aligned} 
  \enorm{(\vec\Pi_h^{k+1} \vec v, \vech\Pi_h^{k} \vec v)}_h^2
  &= | \vec\Pi_h^{k+1}\vec v|^2_{1,h}
    + |(\vec\Pi_h^{k+1}  \vec v, \vech\Pi_h^{k} \vec v)|_{\rm j}^2 
  \le C|\vec v|_{1,h}^2,
 \end{aligned}
 \end{equation} 
 which completes the proof.   
\end{proof}
   
By using the above results, we can prove the discrete inf-sup 
 condition for the bilinear form $b_h$.

\begin{lemma}[Discrete inf-sup condition]
 There exists a constant $\beta>0$ independent of $h$ such that 
 \begin{equation}\label{dinfsup}
 \begin{aligned} 
  \sup_{(\vec v_h,  \vech v_h) \in \V\times \hV}
   \frac{b_h(\vec v_h, \vech v_h; q_h)}
        {\enorm{(\vec v_h, \vech v_h)}} 
   \ge \beta \|q_h\| \quad \forall q_h \in Q_h^{k}.
 \end{aligned}
 \end{equation} 
\end{lemma}
\begin{proof}
 It is well-known that the continuous inf-sup condition holds: 
 there exists $\beta'>0$ such that
 \begin{equation} \label{cinfsup}
  \sup_{\vec v  \in \vec H_0^1(\Omega) }
  \frac{(\dive\vec v, q)_\Omega}
       {|\vec v|_{1,h}} 
  \ge \beta' \|q\| \quad \forall q \in L_0^2(\Omega).    
 \end{equation} 
 For all $q_h \in Q_h^{k} \subset L_0^2(\Omega)$, we have  
 \begin{equation*}
 \begin{aligned} 
  \sup_{(\vec v_h, \vech v_h) \in \V \times \hV}
  \frac{b_h(\vec v_h, \vech v_h; q_h)}
       {\enorm{(\vec v_h, \vech v_h)}_h} 
  &\ge 
  \sup_{\vec v \in \vec H_0^1(\Omega)}
  \frac{b_h(\vec\Pi_h^{k+1}\vec v, \vech\Pi_h^{k} \vec v; q_h)}
       {\enorm{(\vec\Pi_h^{k+1}\vec v, \vech\Pi_h^{k} \vec v)}_h} \\
  &=\sup_{\vec v \in \vec H_0^1(\Omega)}
    \frac{(\dive \vec v, q_h)_\Omega}
         {\enorm{(\vec\Pi_h^{k+1}\vec v, \vech\Pi_h^{k} \vec v)}_h} \\
  & \ge 
    C\sup_{\vec v \in \vec H_0^1(\Omega)}
     \frac{(\dive \vec v, q_h)_\Omega}
          {|\vec v|_{1,h}} \\
  &  \ge \beta' \|q_h\|.               
 \end{aligned}
 \end{equation*}
 The proof is complete.   
\end{proof}

\subsection{A priori error estimates}
 We prove optimal error estimates
 by using the results in the previous section.
 In this section, the stabilization parameter $\tau$ is fixed to be a sufficiently large value.
\begin{theorem}[Energy-norm error estimate] \label{thm:enorm}
 Let $(\vec u, p) \in \vec H_0^1(\Omega) \times L^2_0(\Omega)$ 
 be the exact solution of the Stokes equations \eqref{stokes},
 and let $(\vec u_h, \vech u_h,p_h) \in \V \times \hV 
 \times Q_h^{k}$
 be the solution of the method \eqref{rhdg}. 
 Then we have
 \begin{equation}\label{en-error}
 \begin{aligned}
  &\enorm{(\vec u-\vec u_h,\vec u|_{\Gamma_h}-\vech u_h) }
   + \|p - p_h\|  \\
  & \qquad \le C \left(\inf_{(\vec v_h, \vech v_h)  \in \V \times \hV}
     \enorm{(\vec u - \vec v_h,\vec u|_{\Gamma_h}-\vech v_h) } 
     +  \inf_{q_h \in Q_h^{k}}\| p - q_h\| \right). 
 \end{aligned}
 \end{equation}
 If $(\vec u, p) \in \vec H^{k+2}(\Omega) \times H^{k+1}(\Omega)$,
 we obtain
 \begin{equation} \label{en-error-opt}
    \enorm{(\vec u-\vec u_h,\vec u|_{\Gamma_h}-\vech u_h)}
    + \|p - p_h\| \le  C h^{k+1} \left( |\vec u|_{\vec H^{k+2}(\Omega)} + \| p \|_{H^{k+1}(\Omega)}\right).
 \end{equation}
\end{theorem}
\begin{proof}
 Let  $\vec v_h \in \V$,
 $\vech v_h \in \hV$ and $r_h \in Q_h^{k}$ be arbitrary, and
 set $\vec \eta_h = \vec u_h - \vec v_h$, 
 $\vech \eta_h = \vech u_h - \vech v_h$ and $\delta_h = p_h -q_h$.
 Then we have
 \begin{equation}\label{a} 
 \begin{aligned} 
  a_h(\vec \eta_h, \vech \eta_h; \vec w_h, \vech w_h)    
   +  b_h(\vec w_h, \vech w_h;  \delta_h)  &= F(\vec w_h,\vech w_h)
  & \forall& (\vec w_h, \vech w_h)   \in \V \times \hV,   \\
  b_h(\vec \eta_h, \vech \eta_h; r_h) &= G(r_h) &\forall& r_h \in Q_h^{k},
 \end{aligned}
 \end{equation} 
 where $F :\V \times \hV \longrightarrow \mathbb R$ and
 $G : Q_h^{k} \longrightarrow \mathbb R$ are defined by
 \begin{align*}
  F(\vec w_h,\vech w_h) &= a_h(\vec u - \vec v_h, \vec u|_{\Gamma_h} - \vech v_h; \vec w_h, \vech w_h)    
  +  b_h(\vec w_h, \vech w_h;  p - q_h),\\
  G(r_h) &=   b_h(\vec u - \vec v_h, \vec u|_{\Gamma_h}- \vech u_h, r_h).
 \end{align*}
 By the boundedness of $a_h$ and $b_h$, we can estimate the dual norms of $F$ and $G$ as follows:
 \begin{equation}\label{b} 
  \begin{aligned} 
   \| F \| &= \sup_{(\vec w_h, \vech w_h) 
   \in \V \times \hV}
        \frac{|F(\vec w_h, \vech w_h)|}
             {\enorm{(\vec w_h, \vech w_h)}}\\
   & \le
     C(\enorm{(\vec u-\vec v_h, \vec u|_{\Gamma_h} - \vech v_h)}
       + \|p - q_h\|),\\
   \|G\| 
   &=   \sup_{r_h \in Q_h^{k}} \frac{|G(r_h)|}{\|r_h\|} 
             \le C\enorm{(\vec u-\vec v_h, \vec u|_{\Gamma_h} - \vech v_h)}.
 \end{aligned}
 \end{equation} 
 By the coercivity \eqref{coer}, we have
 \begin{equation}\label{esti-eta} 
 \begin{aligned} 
   C \enorm{(\vec\eta_h, \vech\eta_h)}^2
   &\le a_h(\vec\eta_h, \vech\eta_h; \vec\eta_h, \vech\eta_h) \\
    &= F(\vec\eta_h, \vech\eta_h) - G(\delta_h) \\
    &\le \|F\|\enorm{(\vec\eta_h, \vech\eta_h)} + \|G\|\|\delta_h\|.
    \end{aligned}
    \end{equation} 
   By the discrete inf-sup condition, we deduce that
   \begin{equation}\label{esti-delta} 
   \begin{aligned}
    \beta\|\delta_h\| 
    &\le \sup_{(\vec w_h,\vech w_h) \in \V \times \hV }
    \frac{b(\vec w_h,\vech w_h; \delta_h)}
    {\enorm{(\vec w_h, \vech w_h)}}\\
    &\le \sup_{(\vec w_h,\vech w_h) \in \V \times \hV }
        \frac{F(\vec w_h,\vech w_h)-a_h(\vec \eta_h, \vech \eta_h; \vec w_h, \vech w_h) }
        {\enorm{(\vec w_h, \vech w_h)}} \\
    &\le \|F\| + C\enorm{(\vec \eta_h, \vech \eta_h)}.
 \end{aligned}
 \end{equation} 
 From \eqref{esti-eta} and \eqref{esti-delta}, it follows that
 \[
  \enorm{(\vec\eta_h, \vech\eta_h)} + 
  \|\delta_h\| 
  \le C\left(\enorm{(\vec u-\vec v_h, \vec u|_{\Gamma_h} - \vech v_h)}
   + \|p - q_h\|\right).
 \]
 By the triangle inequality, we obtain the estimate \eqref{en-error}.  
 In addition, using the approximation property, we see that \eqref{en-error-opt} holds.  
\end{proof}
 
We can prove the $L^2$ error estimate of  optimal order  by the Aubin-Nitsche duality argument.	
\begin{theorem}[$L^2$-error estimate] \label{thm:l2norm}
 Let the notation be the same as in Theorem \ref{thm:enorm}.
 If $(\vec u, p) \in \vec H^{k+2}(\Omega) \times H^{k+1}(\Omega)$, then we have
 \begin{equation} \label{l2err}
   \|\vec u -  \vec u_h\|
   \le Ch^{k+2}\left(|\vec u|_{\vec H^{k+2}(\Omega)} + |p|_{H^{k+1}(\Omega)}\right).
 \end{equation}
\end{theorem}
\begin{proof}
 We consider the following adjoint problem:  find $(\vec \psi,  \pi) \in
  (\vec H^2(\Omega)\cap \vec H^1_0(\Omega))  \times (H^1(\Omega) \cap L_0^2(\Omega))$
 such that 
 \begin{equation*}
 \begin{aligned} 
   -\Delta \vec\psi + \nabla \pi &= \vec u - \vec u_h, \\
   \dive \vec\psi &= 0.
 \end{aligned}
 \end{equation*} 
 Note that $|\vec\psi|_{\vec H^2(\Omega)} + |\pi|_{H^1(\Omega)} \le C\|\vec u - \vec u_h\|$.
 From the adjoint consistency \eqref{acons}, 
 the solution of the  problem satisfies   
 \begin{equation}\label{adjoint-prob} 
 \begin{aligned}
  a_h(\vec v_h, \vech v_h; \vec \psi, \vec \psi|_{\Gamma_h})    
   + b_h(\vec v_h, \vech v_h;  \pi)  &= (\vec u - \vec u_h, \vec v_h)_\Omega
  & \forall& (\vec v_h, \vech v_h) 
   \in \V\times \hV,   \\
   b_h(\vec \psi, \vec \psi|_{\Gamma_h}; q_h) &= 0 
  & \forall& q_h \in Q^{k}_h.
 \end{aligned}
 \end{equation} 
 Let ($\vec \psi_h, \vech \psi_h, \pi_h)$ $ \in \V \times \hV \times Q_h^{k}$  be an approximation to $(\vec \psi, \vec \psi|_{\Gamma_h}, \pi)$ satisfying
 \begin{align*}
  \enorm{(\vec\psi - \vec \psi_h, \vec\psi|_{\Gamma_h} - \vech \psi_h)} \le Ch |\vec\psi|_{\vec H^2(\Omega)}, \quad
  \| \pi - \pi_h\| \le Ch |\pi|_{H^1(\Omega)}.
 \end{align*}
 Taking $\vec v_h = \vec u - \vec u_h$
 and $\vech v_h = \vec u|_{\Gamma_h} - \vech u_h$ in \eqref{adjoint-prob},  
 \begin{equation*}
 \begin{aligned} 
 \|\vec u - \vec u_h\|^2
  & = a_h(\vec u -\vec u_h, \vec u|_{\Gamma_h} - \vech u_h; \vec\psi, \vec \psi|_{\Gamma_h})    
    + b_h(\vec u -\vec u_h, \vec u|_{\Gamma_h} - \vech u_h;  \pi) \\
  & = a_h(\vec u -\vec u_h, \vec u|_{\Gamma_h} - \vech u_h; \vec\psi - \vec\psi_h, \vec \psi|_{\Gamma_h}-\vech\psi_h)    \\
  & \qquad +  b_h(\vec u -\vec u_h, \vec u|_{\Gamma_h} - \vech u_h;  \pi - \pi_h) \\
  & \le C \enorm{(\vec u -\vec u_h, \vec u|_{\Gamma_h} - \vech u_h)}
     (\enorm{(\vec\psi - \vec \psi_h, \vec \psi|_{\Gamma_h}-\vech \psi_h)} + \| \pi - \pi_h\|) \\
  & \le C h^{k+2} \left(|\vec u|_{\vec H^{k+2}(\Omega)}+|p|_{H^{k+1}(\Omega)} \right)\|\vec u - \vec u_h\|.
 \end{aligned}
 \end{equation*} 
 Thus we obtain the assertion.  
\end{proof}
 
\section{Relations with the nonconforming Gauss-Legendre  finite element method}

\subsection{The Gauss-Legendre element}
The approximation space of the $(k+1)$-th  Gauss-Legendre element for velocity is defined by 
\begin{equation}\label{} 
 \begin{aligned} 
   \widetilde{\vec V}_h^{k+1} = \{ \vec v_h \in \vec V_h^{k+1} :
     \jump{\Proj_{k} \vec v_h} = 0 \},
 \end{aligned}
\end{equation} 
 where $\jump{\cdot}$ is a jump operator(see \cite{ABCM2002} for example).
 This space is known as the Crouzeix-Raviart\cite{CrRa1973}$(k=0)$,
 Fortin-Soulie\cite{FoSo1983}$(k=1)$ or Crouzeix-Falk \cite{CrFa1989}($k=2$) element.
 Note that $\tvec v_h \in \tV$ is continuous at the $(k+1)$-th order Gauss-Legendre points, 
and thereby $\Proj_{k} \tvec v_h$ is single-valued on $\Gamma_h$.
   
The nonconforming Gauss-Legendre   finite element method
 reads as:
 find $(\vec u_h^*, p_h^*) \in \widetilde{\vec V}_h^{k+1} \times Q_h^{k}$
  such that
 \begin{subequations}\label{gl} 
 \begin{align} 
  (\nabla \vec u_h^*, \nabla \tvec v_h)_{\T_h} - (\dive \tvec v_h, p_h^*)_{\T_h} &= (\vec f, \tvec v_h)_\Omega 
     &  \forall& \vec v_h \in \tV,
       \label{gl1}
        \\
   (\dive \vec u_h^* , q_h)_{\T_h} &= 0 & \forall& q_h \in Q_h^{k}.
   \label{gl2}
 \end{align}
 \end{subequations}   
 The nonconforming method is well-posed for $k = 0,1,2$.
 For $k \ge 3$, it is the case 
   under some assumption on a mesh, see \cite[Lemma 3.1]{BaSt2007}. 
 We here assume that the method \eqref{gl} is well-posed
 for simplicity.
 In our analysis, we will use the following inf-sup condition for
  the Gauss-Legendre element, see also \cite{BaSt2007}.
\begin{theorem}[Discrete inf-sup condition for the Gauss-Legendre element]
 There exists a constant $\tilde\beta>0$ such that
 \begin{equation}
  \sup_{\tvec v_h \in \tvec V_h^{k+1}}
  \frac{(\dive \tvec v_h, p_h)_{\T_h}}{|\tvec v_h|_{1,h}}
  \ge \tilde \beta \| q_h\|
  \quad \forall q_h \in Q_h^{k}.
 \end{equation}
\end{theorem}

\subsection{A relation between the  reduced-order HDG method 
    with the lowest-order approximation
  and the Crouzeix-Raviart element}
 
 It is known that there are relations between HDG methods
 and the conforming or nonconforming finite element methods
 for Poisson's equation.
 The hybrid part the solution of the embedded discontinuous Galerkin(EDG) method\cite{GCS2007}
 with the continuous $P_1$ element
 is identical to the conforming finite element solution on $\Gamma_h$.
 In \cite{OikawaToAppear}, it was proved that the hybrid part of the solution of the reduced method
 using the discontinuous $P_1$--$P_0$ approximation coincides
 with the nonconforming Crouzeix-Raviart  finite element method
 at the mid-points of edges.
 In this section, we discover a relation between the reduced method
 with the $P_1$--$P_0$/$P_0$ approximation  
 and the nonconforming Crouzeix-Raviart  finite element method.
 
 Let $\widetilde{V}_h^1$ be the usual Crouzeix-Raviart finite element space.
 The Crouzeix-Raivart interpolation operator
 for scalar-valued functions, 
 $\Pi_h^* : L^2(\Gamma_h) \to \widetilde{V}_h^1$, is defined as
 \begin{equation}
    \int_e \Pi_h^*  \hat v ds = \int_e \hat v ds   
    \quad \forall e \in \E_h.
 \end{equation}
 For a vector-valued function 
 $\vech v = (\widehat v_1, \cdots, \widehat v_d)^T \in  \vec L^2(\Gamma_h)$, 
 we set
 \[
   \vec\Pi_h^*\vech v :=  (\Pi_h^* \widehat v_1, \cdots, \Pi_h^* \widehat v_d)^T.
 \]
 We are now in a position to prove the relation. 
\begin{theorem}
 Let $(\vec u_h, \vech u_h, p_h) 
 \in \vec V_h^1 \times \vech V_h^0 \times Q_h^0$
 be the solution of the reduced method \eqref{rhdg} with $k=0$ and
 $(\vec u_h^*, p_h^*) \in \tvec V_h^1 \times Q_h^0$ be
 the solution of the nonconforming Crouzeix-Raviart finite element method. 
 Then we have
 \[
  \vec \Pi_h^* \vech u_h = \vec u_h^*, 
  \quad p_h = p_h^*.
 \]
\end{theorem}
\begin{proof}
  Since $\vec\Pi_h^* \vech v_h$ is a piecewise linear polynomial, 
  it follows that, by the Green formula,   
  \begin{align*}
   \ang{\ngrad \vec\Pi_h^* \vech v_h, \vech u_h}_{\dT_h}
   &= \ang{\ngrad \vec\Pi_h^* \vech v_h, \vec\Pi_h^* \vech u_h}_{\dT_h}\\
   &=(\nabla \vec\Pi_h^* \vech v_h, \nabla \vec\Pi_h^* \vech u_h)_{\T_h}.
   \end{align*}
  Choosing $\vec v_h = \vec\Pi_h^* \vech v_h$ in \eqref{rhdg1}
   and noting that $\vech v_h - \vec\Pi_h^* \vech v_h = \vec 0$ 
   at the mid-points of edges, we have
   \begin{equation}\label{cr1} 
    \begin{aligned} 
     (\nabla \vec \Pi_h^* \vech u_h, \nabla \vec\Pi_h^* \vech v_h)_{\T_h}
      - (\dive \vec \Pi_h^* \vech v_h, p_h)_{\T_h} = (\vec f, \vec\Pi_h^* \vech v_h)_\Omega
      \quad \forall \vech v_h \in \vech V_h^0.
    \end{aligned}
   \end{equation} 
   The equation \eqref{rhdg2} can be rewritten as
   \begin{equation}\label{cr2} 
    \begin{aligned} 
    - (\dive \vec u_h, q_h)_{\T_h} - \ang{\vech u_h - \vec u_h, q_h\vec n}_{\dT_h}
    &=  - \ang{\vech u_h , q_h\vec n}_{\dT_h} \\
     &=  - \ang{\vec \Pi_h^*\vech u_h , q_h\vec n}_{\dT_h} \\
     &= - (\dive \vec \Pi_h^* \vech u_h , q_h\vec n)_{\T_h}.
    \end{aligned}
   \end{equation} 
   From \eqref{cr1} and \eqref{cr2},
    we obtain the following equations to determine $\vech u_h$:
   \begin{align}\label{cr3a} 
     (\nabla \vec \Pi_h^* \vech u_h, \nabla \vec\Pi_h^* \vech v_h)_{\T_h}
      - (\dive \vec \Pi_h^* \vech v_h, p_h)_{\T_h} &= (\vec f, \vec\Pi_h^* \vech v_h)_\Omega
      \quad & \forall \vech v_h \in \vech V_h^0, \\
         (\dive \vec \Pi_h^* \vech u_h , q_h\vec n)_{\T_h} &= 0 \quad  &\forall q_h \in Q_h^0,
   \label{cr3b}
   \end{align}
       which is nothing but \eqref{gl} in the case of $k=1$. 
   Therefore we have $\vec \Pi_h^* \vech u_h = \vec u_h^*$ and $p_h = p^*_h$.     
\end{proof}
\begin{remark} From the definition of the Crouzeix-Raviart interpolation,
   we see that 
   \[  \int_e \vech u_h ds 
   =\int_e \vec \Pi_h^* \vech u_h ds 
     = \int_e \vec u_h^* ds,
      \]
       which implies 
    $\vech u_h$ and $\vec u_h^*$ are equal 
     at the mid-point of  $e \in \E_h$.  
\end{remark}
\subsection{The limit case as $\tau$ tends  to infinity}
  We will show that, as the stabilization parameter $\tau$ tends to infinity, 
   the solution of the reduced  method
  converges to the nonconforming Gauss-Legendre   finite element solution
  with rate $O(\tau^{-1})$.
  To do that, we first prove the following key lemma.

 \begin{lemma} \label{keylemma}
     There exists a constant $C>0$ such that, for any $(\vec v_h, \vech v_h) \in \V \times \hV$,
   \begin{equation}
   \inf_{\tvec w_h \in \tV}
   \enorm{(\vec v_h - \tvec w_h, \vech v_h-\Proj_{k}\tvec w_h)}_{h}
   \le C |(\vec v_h, \vech v_h)|_{\rm j}. 
   \end{equation}
 \end{lemma}
     
 \begin{proof}
     Let us define $\vec G(\tV) = \{(\tvec v_h, \Proj_{k}\tvec v_h) : 
     \tvec v_h \in \tV\}$ which is a closed subspace of $\V\times \hV$.
    We consider the quotient space $\V\times \hV/\vec G(\tV)$,
    where the standard quotient norm is given by 
    \[
       \| (\vec v_h, \vech v_h) \|_{\V\times \hV/\vec G(\tV)}
       = \inf_{\tvec w_h \in \tV} 
       \enorm{(\vec v_h - \tvec w_h, \vech v_h - \Proj_{k}\tvec w_h)}_h.
    \] 
    We shall prove that $|(\vec v_h, \vech v_h)|_{\rm j}$ is 
    also a norm on $\V\times \hV/\vec G(\tV)$.
    Suppose that $|(\vec v_h, \vech v_h)|_{\rm j} = 0$,
    then we readily have $\vec v_h \in \tV$ and 
    $\vech v_h = \Proj_{k}\vec v_h$.
    Hence it follows that $(\vec v_h, \vech v_h) \in \vec G(\tV)$,
    which implies that $(\vec v_h, \vech v_h)$ is the zero element
     of $\V\times \hV/\vec G(\tV)$.
     Since any two norms on a finite-dimensional space
      are equivalent to each other, we conclude the assertion.    
       
 \end{proof}

 Let us denote $(\vec u_h^\tau, \vech u_h^\tau, p_h^\tau)$  
  the solution of the reduced method \eqref{rhdg}
   with the stabilization parameter $\tau$. 
   By using Lemma \ref{keylemma}, it can be proved that
     $|(\vec u_h^\tau, \vech u_h^\tau)|_{\rm j}$ converges to zero
     as $\tau \to \infty$ with rate $O(\tau^{-1})$.

 \begin{theorem} 
     \label{thm:tauinf-jump}
     If $\tau$ is sufficiently large, we have
     \begin{align} \label{tauinf-jump-rate}
     & |(\vec u_h^\tau, \vech u_h^\tau)|_{\rm j} 
     \le C\tau^{-1} \|\vec f\|.      
     \end{align}
 \end{theorem}
 
 \begin{proof}
 Let $\tvec v_h \in \tV$ be arbitrary.
Since $\Proj_{k}(\tvec v_h - \Proj_{k} \tvec v_h)$ 
vanishes on $\Gamma_h$,  we have
\begin{equation*}
\begin{aligned}
|(\vec u_h^\tau, \vech u_h^\tau)|_{\rm j, \tau}
&=  |(\vec u_h^\tau - \tvec v_h,
\vech u_h^\tau - \Proj_{k}\tvec v_h)|_{\rm j, \tau} \\
& \le \enorm{(\vec u_h^\tau - \tvec v_h,
    \vech u_h^\tau - \Proj_{k}\tvec v_h)}_{\tau}.
\end{aligned}
\end{equation*}
Therefore,
\begin{equation}
\label{tauinf-rate1}
\begin{aligned}
|(\vec u_h^\tau, \vech u_h^\tau)|_{\rm j, \tau}
&\le  \inf_{\tvec w_h \in \tV}\enorm{(\vec u_h^\tau - \tvec w_h,
    \vech u_h^\tau - \Proj_{k}\tvec w_h)}_{\tau}.
\end{aligned}
\end{equation}
 Taking $(\vec v_h,  \vech v_h) = (\tvec v_h, \Proj_{k}\tvec v_h)$ in \eqref{rhdg}, we have
 \begin{align}
 (\nabla \vec u_h^\tau, \nabla \tvec v_h)_{\T_h} 
 + \ang{\ngrad \tvec v_h, \vech u_h^\tau - \vec u_h^\tau}_{\dT_h}
  -  (\dive \tvec v_h, p_h^\tau)_{\T_h}
  = (\vec f, \tvec v_h)_\Omega.
 \end{align}
 Subtracting this from \eqref{rhdg1} by taking
  $(\vec v_h, \vech v_h) = (\vec u^\tau_h, \vech u^\tau_h)$ gives
 \begin{equation}
 \label{tauinf-rate2}
  \begin{aligned}
  |(\vec u_h^\tau, \vech u_h^\tau)|_{\rm j, \tau}^2
  =& (\vec f, \vec u_h^\tau - \tvec v_h)_\Omega 
      -(\nabla \vec u_h^\tau, \nabla(\vec u_h^\tau - \tvec v_h))_{\T_h}\\
  &   -\ang{\ngrad \vec u_h^\tau, \vech u_h^\tau - \vec u_h^\tau}_{\dT_h}
     -\ang{\ngrad (\vec u_h^\tau-\tvec v_h), \vech u_h^\tau - \vec u_h^\tau}_{\dT_h}\\
  & + (\dive (\vec u_h^\tau - \tvec v_h), p_h^\tau)_{\T_h}
   +  \ang{\vech u_h^\tau - \vec u_h^\tau, p_h^\tau\vec n}_{\dT_h}.
  \end{aligned}
  \end{equation}
%
\newcommand{\tmp}{\vec u_h^\tau - \tvec w_h,  \vech u_h^\tau - \Proj_{k}\tvec w_h}
Here we note that the energy norm is stronger than the $L^2$ norm(see \cite{OikawaToAppear} for the proof): 
     \begin{equation} \label{tauinf-rate3}
     \|\vec z_h\|
     \le C\enorm{(\vec z_h, \vech z_h)}_h
     \qquad \forall (\vec z_h, \vech z_h) \in \V \times \hV,
     \end{equation}
     and that by the standard argument, we see that $\vec u_h^\tau$ and $p_h^\tau$ are uniformly bounded in
     the energy and $L^2$ norms, respectively, i.e.
     \[
       \enorm{\vec u_h^\tau}_{h} + \|p_h^\tau\| \le C\|\vec f\|,
     \] 
     where $C$ is independent of $\tau$.
 The terms on the right-hand side of \eqref{tauinf-rate2} are
 bounded as follows:
 \newcommand{\tmpdiff}{(\vec u_h^\tau - \tvec v_h, \vech u_h^\tau -\Proj_{k}\tvec v_h)}
 \begin{align*}
   |(\vec f, \vec u_h^\tau - \tvec v_h)_\Omega|
  & \le \|\vec f\| \|\vec u_h^\tau - \tvec v_h\|\\
 &  \le C\|\vec f\| \enorm{\tmpdiff}_{h}, 
   \\
   |(\nabla \vec u_h^\tau, \nabla(\vec u_h^\tau - \tvec v_h))_{\T_h}|
 &  \le \|\nabla \vec u_h^\tau\| \|\nabla(\vec u_h^\tau - \tvec v_h)\| \\
 &  \le C\|\vec f\| \enorm{\tmpdiff}_{h},\\
 |\ang{\ngrad \vec u_h^\tau, \vech u_h^\tau - \vec u_h^\tau}_{\dT_h}| 
 & \le  C\|\nabla \vec u_h^\tau\| |(\vec u_h^\tau, \vech u_h^\tau)|_{\rm j} \\
 &\le C\|\vec f\| |(\vec u_h^\tau, \vech u_h^\tau)|_{\rm j},\\
 |\ang{\ngrad (\vec u_h^\tau-\tvec v_h), \vech u_h^\tau - \vec u_h^\tau}_{\dT_h}|
& \le C\|\nabla(\vec u_h^\tau-\tvec v_h)\| 
  |(\vec u_h^\tau, \vech u_h^\tau)|_{\rm j} \\
  &\le C\enorm{\tmpdiff}_h 
  |(\vec u_h^\tau, \vech u_h^\tau)|_{\rm j} \\
|(\dive (\vec u_h^\tau - \tvec v_h), p_h^\tau)_{\T_h}| 
&\le \|\nabla(\vec u_h^\tau - \tvec v_h)\| \|p_h^\tau\| \\
&\le C\enorm{\tmpdiff}_h\|\vec f\|
\\
|\ang{\vech u_h^\tau - \vec u_h^\tau, p_h^\tau\vec n}_{\dT_h}|
&\le C|(\vec u_h^\tau, \vech u_h^\tau)|_{\rm j} 
 \|p_h^\tau\| \\
 & \le C|(\vec u_h^\tau, \vech u_h^\tau)|_{\rm j}\|\vec f\|.
 \end{align*}
 From the above estimates, we obtain, by using Lemma \ref{keylemma},
 \begin{align*}
 |(\vec u_h^\tau, \vech u_h^\tau)|_{\rm j, \tau}^2
 &\le C(\|\vec f\| + |(\vec u_h^\tau, \vech u_h^\tau)|_{\rm j})
  \cdot \inf_{\tvec w_h \in \tV}\enorm{(\vec u_h^\tau - \tvec w_h, \vech u_h^\tau -\Proj_{k}\tvec w_h)}_h \\
&  \le C_*(\|\vec f\| + |(\vec u_h^\tau, \vech u_h^\tau)|_{\rm j})
   |(\vec u_h^\tau, \vech u_h^\tau)|_{\rm j}.
 \end{align*}
 Since we can assume $\tau > C_*$,  it finally follows that
 \[
 |(\vec u_h^\tau, \vech u_h^\tau)|_{\rm j}
 \le \frac{C}{\tau - C_*} \|\vec f\|
 \le C\tau^{-1} \|\vec f\|.
 \] 
   The proof is complete.  
   \end{proof}

In the following theorem, we prove that $(\vec u_h^\tau, p_h^\tau)$ 
converges to $(\vec u_h^*,  p_h^*)$ with the same rate  as $|(\vec u_h^\tau, \vech u_h^\tau)|_{\rm j}$, namely $O(\tau^{-1})$. 
   \begin{theorem} \label{thm:tauinf-conv}
    If $\tau$ is sufficiently large, we have
       \begin{equation}
       |\vec u_h^* - \vec u_h^\tau|_{1,h}
       + \| p_h^* - p_h^\tau\| \le C\tau^{-1} \|\vec f\|.
       \end{equation}
    \end{theorem}  
    \begin{proof}
    Choosing $\vec v_h = \tvec v_h \in \tvec V_h^k$ and
    $\vech v_h = \Proj_{k}\tvec v_h$
    in \eqref{rhdg}, we have 
    \begin{subequations} \label{rhdg-gl}
        \begin{align}
        (\nabla \vec u_h^\tau, \nabla \tvec v_h)_{\T_h}
        +\ang{\ngrad\tvec v_h, \vech u_h^\tau - \vec u_h^\tau}_{\dT_h}
        -(\dive \tvec v_h, p_h^\tau)_{\T_h}
        &  = (\vec f, \tvec v_h)_\Omega 
        & \forall& \tvec v_h \in \tV,
        \\
        (\dive \vec u_h^\tau, q_h)_{\T_h}
        +\ang{\vech u_h^\tau - \vec u_h^\tau, q_h \vec n}_{\dT_h}
        &= 0 & \forall& q_h \in Q_h^{k}. 
        \end{align}
    \end{subequations}   
    Let us denote $\vec\xi_h^\tau = \vec u_h^* - \vec u_h^\tau$,
    $\vech\xi_h^\tau = \Proj_{k} \vec u_h^* - \vech u_h^\tau$
     and $\delta_h^\tau = p_h^* - p_h^\tau$.
     Subtracting \eqref{rhdg-gl} from \eqref{gl} leads to
     \begin{subequations}
     
        \begin{align}
        (\nabla \vec \xi_h^\tau, \nabla \tvec v_h)_{\T_h}
        -\ang{\ngrad\tvec v_h, \vech u_h^\tau - \vec u_h^\tau}_{\dT_h}
        -(\dive \tvec v_h, \delta_h^\tau)_{\T_h}
        &  = 0
        & \forall& \tvec v_h \in \tV,
        \label{eq-diff-a}
        \\
        (\dive \vec \xi_h^\tau, q_h)_{\T_h}
        +\ang{\vech u_h^\tau - \vec u_h^\tau, q_h \vec n}_{\dT_h}
        &= 0 & \forall& q_h \in Q_h^{k}.
        \label{eq-diff-b} 
        \end{align}  
     \end{subequations}  
    Let $\tvec w_h \in \tV$ be arbitrary, then we have
    \begin{align*}
     (\nabla \vec\xi_h^\tau,  \nabla \vec\xi_h^\tau)_{\T_h}
      & = (\nabla(\vec u_h^* - \tvec w_h),  \nabla \vec\xi_h^\tau)_{\T_h}
      + (\nabla(\tvec w_h - \vec u_h^\tau),  \nabla \vec\xi_h^\tau)_{\T_h}\\
      &=: I + II.
    \end{align*}
    By \eqref{eq-diff-a} with $\tvec v_h = \vec u_h^* - \tvec w_h$, we get
    \begin{align*}
     I &= \ang{\ngrad (\vec u_h^* - \tvec w_h), \vech u_h^\tau - \vec u_h^\tau}_{\dT_h}
          + (\dive(\vec u_h^* - \tvec w_h), \delta_h^\tau)_{\T_h} \\
       &= \ang{\ngrad \vec \xi_h^\tau, \vech u_h^\tau - \vec u_h^\tau}_{\dT_h}
          +\ang{\ngrad(\vec u_h^\tau - \tvec w_h),  \vech u_h^\tau - \vec u_h^\tau}_{\dT_h}\\
       &\quad + (\dive\vec \xi_h^\tau, \delta_h^\tau)_{\T_h}
          +  (\dive(\vec u_h^\tau - \tvec w_h), \delta_h^\tau)_{\T_h} \\
       & =: I_1 + I_2 + I_3 + I_4.   
    \end{align*}
    These terms can be bounded as follows:
    \begin{align*}
    |I_1| &\le C|\vec \xi_h^\tau|_{1,h} \jumpuh, \\
    |I_2| & \le C|\vec u_h^\tau - \tvec w_h |_{1,h} \jumpuh,\\
    |I_3| & = |-\ang{\vech u_h^\tau - \vec u^\tau_h, \delta_h^\tau \vec n}_{\dT_h}| 
            \qquad \text{ (by \eqref{eq-diff-b})} \\
          & \le C \jumpuh \|\delta_h^\tau\|,\\
    |I_4| & \le C |\vec u_h^\tau - \tvec w_h|_{1,h} \|\delta_h^\tau\|.
    \end{align*}
    Therefore, we have
    \begin{align*}
    |I| &\le   C(|\vec\xi_h^\tau|_{1,h}+\|\delta_h^\tau\|)\jumpuh
          +C(\jumpuh+\|\delta_h^\tau\|)|\vec u_h^\tau - \tvec w_h|_{1,h}.
    \end{align*}
    Taking the infimum with respect to $\tvec w_h$ and using Lemma \ref{keylemma}, we have
     \begin{equation} \label{esti-tau-xi-I}
    \begin{aligned}
     |I|
        &\le C 
     \left( |(\vec u_h^\tau, \vech u_h^\tau)|_{\rm j} 
     + |\vec \xi_h^\tau|_{1,h} + \|\delta_h^\tau\|
     \right) 
     |(\vec u_h^\tau, \vech u_h^\tau)|_{\rm j}.
    \end{aligned}
         \end{equation}
    On the other hand, by the inf-sup condition for the Gauss-Legendre element,
    we have
    \begin{equation} \label{esti-tau-delta}
    \begin{aligned}
    \tilde \beta \| \delta_h^\tau \|
     &\le \sup_{\tvec v_h \in  \tV} 
     \frac{(\dive\tvec v_h,  \delta_h^\tau)_{\T_h}}
     {|\tvec v_h|_{1,h}} \\
     &=  \sup_{\tvec v_h \in  \tV} 
     \frac{ (\nabla \vec \xi_h^\tau, \nabla \tvec v_h)_{\T_h}
          - \ang{\ngrad\tvec v_h, \vech u_h^\tau - \vec u_h^\tau}_{\dT_h}
         }{|\tvec v_h|_{1,h}} \\
     &\le
      |\vec \xi_h^\tau|_{1,h}
       + C|(\vec u_h^\tau, \vech u_h^\tau)|_{\rm j}.
    \end{aligned}
         \end{equation}
    Combining \eqref{esti-tau-xi-I} with \eqref{esti-tau-delta}  gives us  
           \begin{equation} \label{esti-I}
           \begin{aligned}
           |I| 
           &\le C 
           \left(\jumpuh+|\vec \xi_h^\tau|_{1,h}\right)\jumpuh.
           \end{aligned}
           \end{equation}
    The term $II$ is estimated as follows  by using Lemma \ref{keylemma}:
    \begin{equation}\label{esti-II}
    \begin{aligned}
     |II| &\le
      \inf_{\tvec w_h \in \tV}\enorm{(\vec u_h^\tau - \tvec w_h, \vech u_h^\tau -\Proj_{k}\tvec w_h)}_h
     \cdot  |\vec \xi_h^\tau|_{1,h} \\
     &\le C\jumpuh |\vec \xi_h^\tau|_{1,h}. 
    \end{aligned}
    \end{equation}
    From \eqref{esti-I} and \eqref{esti-II}, it follows that
    \[
     |\vec \xi_h^\tau|_{1,h}^2 \le 
     C 
           \left(\jumpuh+|\vec \xi_h^\tau|_{1,h}\right)\jumpuh.
    \]
    By Young's inequality, we get
    $
     |\vec\xi_h^\tau|_{1,h} \le C|(\vec u_h^\tau, \vech u_h^\tau)|_{\rm j}.
    $
    Using \eqref{esti-tau-delta} again, we have
    $ 
       \|\delta_h^\tau\| \le C|(\vec u_h^\tau, \vech u_h^\tau)|_{\rm j}.
    $
    By Theorem \ref{thm:tauinf-jump}, we obtain the assertion.  
\end{proof}

\section{Numerical results} 
As a test problem, we consider the case of $\Omega = (0,1)^2$  and
\begin{align*}
 \vec f(x,y) = (4\pi^2\sin(2\pi y), 4\pi^2\sin(2\pi x) (-1+4\cos(2\pi y)))^T.
\end{align*}
The source term is chosen so that the exact solution $(\vec u,p)$ is  
\begin{align*}
 \vec u(x,y) &= (2\sin^2(\pi x)\sin(\pi y), -2\sin(\pi x)\sin^2(\pi y))^T, \\
p(x,y) &= 4\pi \sin(2\pi x)\sin(2\pi y).
\end{align*}

 We carry out numerical computations 
  to examine the convergence rates of the reduced method.
  We employ unstructured triangular meshes and the $P_{k+1}$--$P_{k}$/$P_{k}$ approximations 
  for $k=0,1,2$.
 The convergence history is shown at Table \ref{tb:chist}, and the convergence diagrams are 
  displayed in Figure \ref{fig:cdiagrams}.
  From these results, we observe that
  the convergence orders are optimal in all the cases,
   which agrees with Theorems \ref{thm:enorm} and \ref{thm:l2norm}.

 \begin{table}[h!]
 \caption{Convergence history of the reduced method.}
 \center
 \bgroup
 \renewcommand{\arraystretch}{1.2}
 \begin{tabular}{cccccccc} \hline \hline
 Degree & Mesh size & 
      \multicolumn{2}{c}{$\|\vec u - \vec u_h\|$}  & \multicolumn{2}{c}{$\|\nabla (\vec u - \vec u_h)\|$} & \multicolumn{2}{c}{$\|p - p_h\|$}   \\
$k$ & $h$   & Error & Order & Error & Order & Error & Order \\ \hline
0   & 0.2634  & 1.789E-01 &  -- & 3.672E+00 & -- & 2.307E+00 & --  \\
    & 0.1414   & 4.938E-02 & 2.07 & 1.901E+00 & 1.06  & 1.249E+00 & 0.99  \\
    & 0.0701  & 1.200E-02 & 2.01 & 9.340E-01 & 1.01  & 6.054E-01 & 1.03  \\
    & 0.0373  & 2.907E-03 & 2.25 & 4.596E-01 & 1.13  & 2.925E-01 & 1.15  \\ \hline
1   & 0.2634  & 1.038E-02 &  -- & 8.486E-01 & -- & 4.520E-01 & --  \\
    & 0.1414   & 1.639E-03 & 2.97 & 2.290E-01 & 2.11  & 1.259E-01 & 2.06  \\
    & 0.0701   & 1.815E-04 & 3.13 & 5.003E-02 & 2.17  & 2.768E-02 & 2.16  \\
    & 0.0373  & 2.295E-05 & 3.28 & 1.254E-02 & 2.19  & 6.859E-03 & 2.20  \\ \hline
2  & 0.2634   & 1.338E-03 & --  & 1.439E-01 & --  & 3.105E-02 & -- \\
    & 0.1415 & 8.656E-05 & 4.40 & 1.900E-02 & 3.26  & 3.920E-03 & 3.33  \\
    & 0.0701  & 5.089E-06 & 4.03 & 2.298E-03 & 3.01  & 4.310E-04 & 3.14  \\
    & 0.0373  & 3.168E-07 & 4.40 & 2.858E-04 & 3.31  & 4.956E-05 & 3.43  \\ 
\hline \hline
\end{tabular}
\egroup
\label{tb:chist}
\end{table}

\begin{figure}[b]
\center
\begin{minipage}{150pt}
\includegraphics[width=150pt]{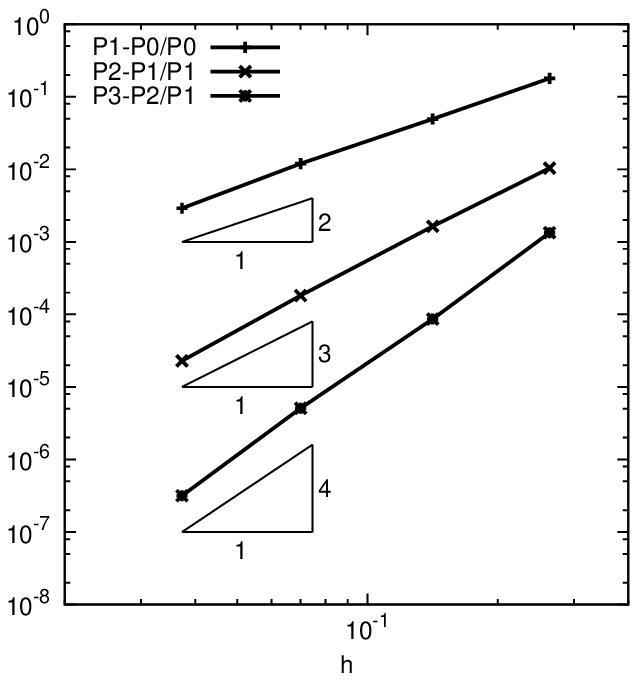}
\begin{center}(a) $L^2$-error of velocity \end{center}
\end{minipage}
\begin{minipage}{150pt}
\includegraphics[width=150pt]{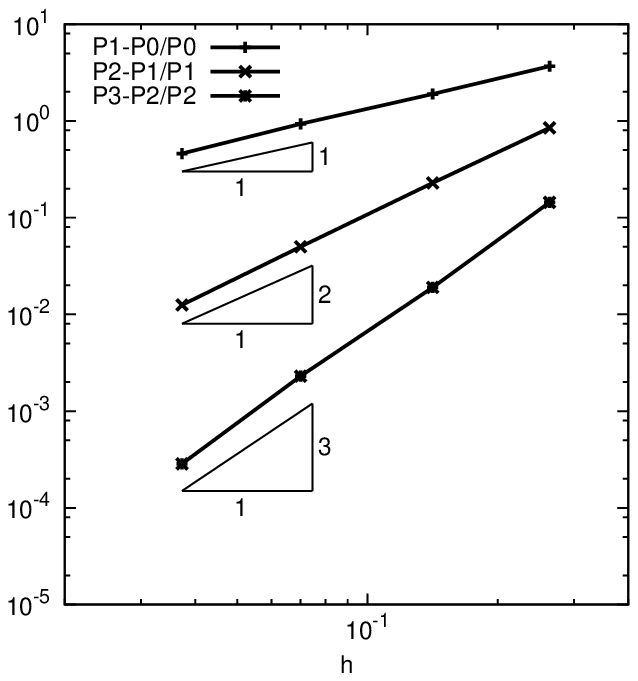}
\begin{center}(b) $H^1$-error of velocity \end{center}
\end{minipage}
\bigskip
\par 
\begin{minipage}{150pt}
\includegraphics[width=150pt]{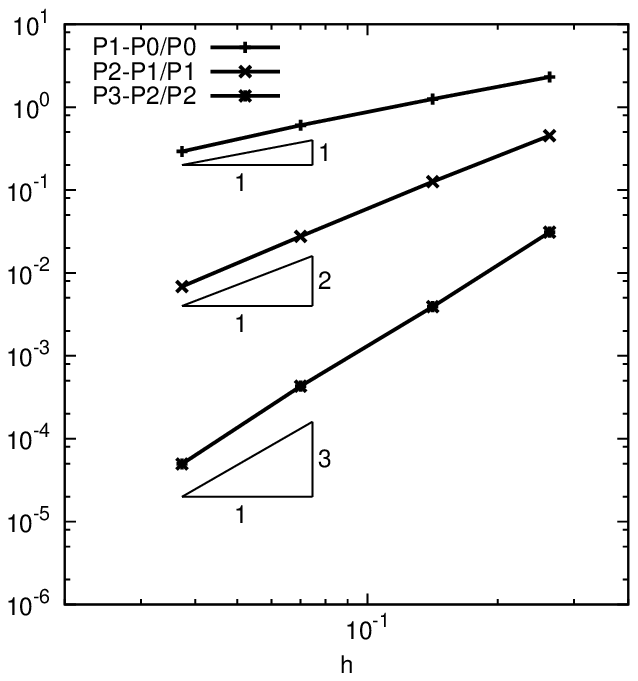}
\begin{center}(c) $L^2$-error of pressure \end{center}
\end{minipage}
\caption{Convergence diagrams of the reduced method.}
\label{fig:cdiagrams}

\end{figure}

  \bibliographystyle{abbrv}
  \bibliography{ref}  
\end{document}